\documentclass[a4paper,12pt]{article}
    \usepackage[top=2.5cm,bottom=2.5cm,left=2.5cm,right=2.5cm]{geometry}
    \usepackage{cite, amsmath, amssymb}
    \usepackage[margin=1cm,%
                font=small,%
                format=hang,%
                labelsep=period,%
                labelfont=bf]{caption}
\usepackage{graphicx,url}
\newtheorem{theorem}{Theorem}
\newtheorem{theo}{Theorem}
\newtheorem{lemma}[theo]{Lemma}

\newenvironment{proof}{\par \noindent \textbf{Proof. }}{\QED \par \bigskip \par}
\newcommand{\QED}{\hfill$\square$}

\allowdisplaybreaks[4]

\begin{document}

\baselineskip=0.35in

\vspace*{0cm}

\begin{center}
{\Large \bf On the Maximum Atom-Bond Sum-Connectivity Index of Graphs}

\vspace{6mm}

{\large \bf Tariq Alraqad}, {\large \bf Hicham Saber}, {\large \bf Akbar Ali\footnote{Corresponding author.}}, {\large \bf Abeer M. Albalahi}

\vspace{6mm}

\baselineskip=0.25in

{\it Department of Mathematics, Faculty of Science,\\ University of Ha\!'il, Ha\!'il, Saudi Arabia}\\[3mm]
{\tt t.alraqad@uoh.edu.sa, Hicham.saber7@gmail.com, akbarali.maths@gmail.com, a.albalahi@uoh.edu.sa}

\vspace{8mm}

\end{center}

%\vspace{3mm}

\begin{abstract}
\noindent
The atom-bond sum-connectivity (ABS) index of a graph $G$ with edges $e_1,\cdots,e_m$ is the sum of the numbers $\sqrt{1-2(d_{e_i}+2)^{-1}}$ over $1\le i \le m$, where $d_{e_i}$ is the number of edges adjacent with $e_i$. In this paper, we study the maximum values of the ABS index over graphs with given parameters. More specifically, we determine the maximum ABS index of connected graphs of a given order and with a fixed (i) minimum degree, (ii) maximum degree, (iii) chromatic number, (iv) independence number, (v) number of pendent vertices. We also characterize the graphs attaining the maximum ABS values in all of these classes.
\end{abstract}


\baselineskip=0.30in

\section{Introduction}
In this paper, just finite and simple graphs are taken into account.
The sets of vertices and edges of a graph $G$ are denoted, respectively, by $V(G)$ and $E(G)$.
The degree of a vertex $v\in V(G)$ is indicated by $d_v(G)$ or just $d_v$ if the graph being discussed is unambiguous.
We utilize the conventional notation and nomenclature of (chemical) graph theory, and we refer readers to the relevant books, for example, \cite{ga,gc}.

Chemical graph theory is a field in which chemical structures are modeled  by graphs. The atoms and bonds are replaced by vertices and edges, respectively. In this way, it is possible to use the concepts of graph theory to study the chemical structures. Graph invariants that adopt quantitative values are widely termed as topological indices in chemical graph theory.

The connectivity index  (or the Randi\'c index) \cite{Gutman-13}, a well-known topological index, was devised in the 1970s by the chemist Milan Randi\'c under the name ``branching index'' \cite{g1}.  Soon after its discovery, the connectivity index quickly found a variety of uses \cite{g2,g3,g4} in chemistry and consequently it become one of the most applied and well-researched index.
For a graph $G$, the connectivity index is defined as
\[
R(G) = \sum_{vw\in E(G)} \frac{1}{\sqrt{d_vd_w}}.
\]

The Randi\'c index has be
modified in several ways. Here, we mention two topological indices which were introduced by taking into consideration the definition of the Randi\'c index, namely the  ``sum-connectivity (SC) index'' \cite{g5} and the ``atom-bond connectivity (ABC) index'' \cite{g7}. These indices have the following definitions for a graph $G$:
\[
SC(G) = \sum_{vw\in E(G)} \frac{1}{\sqrt{d_v+d_w}}
\]
and
\[
ABC(G) = \sum_{wv\in E(G)} \sqrt{\frac{d_v+d_w-2}{d_vd_w}}.
\]
Detail regarding the mathematical aspects of the SC and ABC indices may be found in the review papers
\cite{Ali-19} and \cite{g9}, respectively.

By utilizing the definitions of the ABC and SC indices, a novel topological index -- the atom-bond sum-connectivity (ABS) index -- has recently been proposed in \cite{g10}.
For a graph $G$, this index is defined as
$$
ABS(G) = \sum_{uv\in E(G)} \left(\frac{d_u+d_v-2}{d_u+d_v}\right)^{\frac{1}{2}}\,.
$$
In the paper \cite{g10},  graphs possessing the maximum and minimum values of the ABS index were characterized over the classes of graphs and (chemical) trees of a given order; such kind of extremal results regarding unicyclic graphs were found in \cite{ABS-EJM}, where also chemical applications of the ABS index were reported. The paper \cite{Alraqad-arXiv} is concerned with the problems of determining graphs attaining the minimum ABS index among all trees of a fixed order and/or a given number of pendent vertices; see also \cite{Maitreyi-arXiv} where one of these two problems is attacked independently.

A pendent vertex in a graph is a vertex of degree $1$. The least number of colors required to color the vertices of a graph, so that every two adjacent vertices have different colors, is termed as the chromatic number. A subset $S$ of the vertex set of $G$ is said to be independent if every pair of vertices of $S$ are non-adjacent in $G$. The maximum number among cardinalities of all independent sets of $G$ is known as the independence number of $G$ and it is denoted by $\alpha(G)$. We denote by $\Upsilon_{n,\delta}$,  $\Psi_{n,\Delta}$, $\Gamma_{n,p}$, $\Pi_{n,\chi}$, and $\Sigma_{n,\alpha}$ the classes of graphs of order $n$ and minimum degree $\delta$, maximum degree $\Delta$, $p$ pendent vertices, fixed chromatic number $\chi$, and fixed independence number of $\alpha$ respectively.   In this paper, we aim to characterize the graphs having the maximum values of the ABS index over the classes, $\Upsilon_{n,\delta}$, $\Psi_{n,\Delta}$, $\Gamma_{n,p}$, $\Pi_{n,\chi}$, and $\Sigma_{n,\alpha}$.

\section{Results}\label{Sec2}

Throughout this section, we consider only connected graphs. To prove our results, we need few technical lemmas.

\begin{lemma}[see \cite{g10}]\label{lem-AZI+e}
Let $G$ be a graph. If $u$ and $v$ are non-adjacent vertices in $G$ then $ABS(G + uv) > ABS(G)$.
\end{lemma}

\begin{lemma}\label{l1}
Let $$f(x,y)=\left(\dfrac{x+y-2}{x+y}\right)^{\frac{1}{2}},$$
where $\min\{x,y\}\geq 1$. For every positive real number $s$, define the function $g_s(x,y)=f(x+s,y)-f(x,y)$. Then $f$ is strictly increasing in $x$ and in $y$. The function $g_s$ is strictly decreasing and convex in $x$ and in $y$.

\end{lemma}

\begin{proof}
The first and second partial derivatives of $f$ with respect to $x$ and $y$ are  calculated as
$$\frac{\partial f}{\partial x}(x,y)=\frac{\partial f}{\partial y}(x,y)=(x+y-2)^{-\frac{1}{2}}(x+y)^{-\frac{3}{2}},$$
$$\frac{\partial^2 f}{\partial x^2}(x,y)=\frac{\partial^2 f}{\partial y^2}(x,y)=-\frac{1}{2}(x+y-2)^{-\frac{3}{2}}(x+y)^{-\frac{3}{2}}-\frac{3}{2}(x+y-2)^{-\frac{1}{2}}(x+y)^{-\frac{5}{2}}.$$
Clearly for  $x>1$,  $\frac{\partial f}{\partial x}(x,y)>0$. Thus  $f$ is strictly increasing in $x$ and in $y$.
Since $\frac{\partial^2 f}{\partial x^2}(x,y)<0$, whenever $x>1$, we get $\frac{\partial f}{\partial x}(x,y)$ is strictly decreasing in $x$ when $x\geq1$. So $$\frac{\partial g_s}{\partial x}(x,y)=\frac{\partial f}{\partial x}(x+s,y)-\frac{\partial f}{\partial x}(x,y)<0\quad \text{when \ } x\geq 1,$$ and thus $g_s(x,y)=f(x+s,y)-f(x,y)$, is strictly decreasing  in $x$ when $x\geq 1$. Additional, $\frac{\partial^2 f}{\partial x^2}(x,y)$ is strictly increasing when $x\geq 1$. So  $$\frac{\partial^2 g_s}{\partial x^2}(x,y)=\frac{\partial^2 f}{\partial x^2}(x+s,y)-\frac{\partial^2 f}{\partial x^2}(x,y)>0,$$ and hence $g_s(x,y)$, is convex in $x$ when $x\geq 1$.
\end{proof}

\begin{lemma}\label{l2}
Let $M$ and $N$ be real numbers satisfying $1\leq M\leq N$. Then for every positive real number $s$, the function $h_s(x)=g_s(x,N)-g_s(x,M)$ is increasing in $x$ when $x\geq 1$.
\end{lemma}

\begin{proof}
When $x\geq 1$, we have $\frac{\partial^2 f}{\partial x \partial y}$ is strictly increasing in $x$. So
$$\frac{\partial^2 g_s}{\partial x \partial y}(x,y)=\frac{\partial^2 f}{\partial x \partial y}(x+s,y)-\frac{\partial^2 f}{\partial x \partial y}(x,y)>0.$$
Thus $\frac{\partial g_s}{\partial x}$ is increasing in $y$, and hence $$h_s'(x)=\frac{\partial g_s}{\partial x}(x,N)-\frac{\partial g_s}{\partial x}(x,M)>0.$$ Therefore $h_s(x)$ is increasing in $x$ when $x\geq 1$.
\end{proof}

For $n\geq2$ and $\delta\geq1$, denote by $D_{n,\delta}$ the graph obtained by joining a new vertex to exactly $\delta$ vertices from $K_{n-1}$. The next theorem gives a sharp upper bound on the $ABS$ value of all graphs in the class $\Upsilon_{n,\delta}$. A graph $G$ is called $r-$regular if $d_v=r$ for all $v\in V(G)$, and $G$ is called nearly $r-$regular if $G$ has one vertex of degree $r-1$ and $n-1$ vertices of degree $r$. Nearly $r-$regular graphs are also known as $(r,r-1)-$quasi-regular graphs.

\begin{lemma}\label{l02}{\rm\cite{ca20}} Let $n$ and $r$ be integers such that $2\leq r<n$.
\begin{enumerate}
    \item If $nr$ is even then there is a connected $r-$regular graph of order $n$.
    \item If $nr$ is odd then there is a connected nearly $r-$regular graph of order $n$.
\end{enumerate}
\end{lemma}

\begin{theorem}\label{t0}
Let $n\geq2$ and $\delta\geq1$. If $G\in \Upsilon_{n,\delta}$ then
\begin{align}\label{eq00}
ABS(G)\leq&
\delta \sqrt{\frac{n+\delta-3}{n+\delta-1}}+
\frac{1}{2}\delta(\delta-1)\sqrt{\frac{n-2}{n-1}}+
\delta(n-\delta-1)\sqrt{\frac{2n-5}{2n-3}}\nonumber\\[2mm]
&+\frac{1}{2}(n-\delta-1)(n-\delta-2)\sqrt{\frac{n-3}{n-2}}
\end{align}
with equality if and only if $G\cong D_{n,\delta}$.
\end{theorem}
\begin{proof}
Let $G^{\ast}\in \Upsilon_{n,\delta}$ be having the maximum $ABS$. Let $u\in V(G)$ such that $d_u=\delta$. Assume there are two nonadjacent vertices  $v,w\in V(G)\setminus \{u\}$. Then $G'=G^{\ast}+\{vw\}\in \Upsilon_{n,\delta}$ and $ABS(G')>ABS(G^{\ast})$, a contradiction. Therefore $G^{\ast}\cong D_{n,\delta}$.
\end{proof}

\begin{theorem}\label{t01}
Let $n\geq2$ and $\Delta\geq1$ be integers.
\begin{enumerate}
\item\label{t01p1} If $\Delta n$ is even, then for each $G\in \Psi_{n,\Delta}$,
\begin{equation}\label{eq011}
ABS(G)\leq\frac{n}{2}\sqrt{\Delta(\Delta-1)}
\end{equation}
with equality if and only if $G$ is $\Delta-$regular.
\item\label{t01p2} If $\Delta n$ is odd, then for each $G\in \Psi_{n,\Delta}$,
\begin{equation}\label{eq012}
ABS(G)\leq(\Delta-1)\sqrt{\frac{2\Delta-3}{2\Delta-1}}+\frac{\Delta n-2\Delta+1}{2}\sqrt{\frac{\Delta-1}{\Delta}}
\end{equation}
with equality if and only if $G$ is nearly $\Delta-$regular.
\end{enumerate}
\end{theorem}
\begin{proof}
(\ref{t01p1}) Since$\Delta n$ is even, by Lemma \ref{l02}, there is a $\Delta-$regular graph of order $n$.  Let $G\in \Psi_{n,\Delta}$. Then for each  $uv\in E(G)$, we have $f(d_u,d_v)\leq f(\Delta,\Delta)$. Additionally, we have
$$|E(G)|=\frac{1}{2}\sum_{v\in V(G)}d_v\leq \frac{\Delta n}{2}.$$
Thus $$ABS(G)=\sum_{uv\in E(G)}f(d_u,d_v)\leq \frac{\Delta n}{2}f(\Delta,\Delta)=\frac{n}{2}\sqrt{\Delta(\Delta-1)}.$$
Moreover, the equality holds if and only if $|E(G)|=\frac{\Delta n}{2}$ and for all $uv\in E(G)$, $f(d_u,d_v)=f(\Delta,\Delta)$. Consequently, the equality holds if and only if  $G$ is $\Delta-$regular.

(\ref{t01p2}) Since   $\Delta n$ is odd, by Lemma \ref{l02}, there is a nearly $\Delta-$regular graphs of order $n$.  Let
$G\in \Psi_{n,\Delta}$. Clearly $G$ is not $\Delta-$regular because $\Delta n$ is odd. Let $v\in V(G)$ such that $d_v<\Delta$ and let $M$ be the set of all edges incident with $v$. Then we have for each  $u\in N(v)$, $f(d_u,d_v)\leq f(\Delta,\Delta-1)$ and for each $uw\in E(G)\setminus M$, $f(d_u,d_w)\leq f(\Delta,\Delta)$. Additionally,
$$|E(G)|=\frac{1}{2}\sum_{u\in V(G)}d_u\leq \frac{d_v+\Delta (n-1)}{2}.$$
Thus
\begin{align}ABS(G)=&\sum_{u\in N(v)}f(d_u,d_v)+\sum_{uw\in E(G)\setminus M}f(d_u,d_w)\\[2mm]
\leq & d_vf(\Delta,\Delta-1)+\left(\frac{d_v+\Delta (n-1)}{2}-d_v\right)f(\Delta,\Delta)\\[2mm]
= & d_v\left(\sqrt{\frac{2\Delta-3}{2\Delta-1}}-\sqrt{\frac{\Delta-1}{4\Delta}}\right)+\frac{\Delta (n-1)}{2}\sqrt{\frac{\Delta-1}{\Delta}}.
\end{align}
Since $\sqrt{\frac{2\Delta-3}{2\Delta-1}}-\sqrt{\frac{\Delta-1}{4\Delta}}>0$ and $d_v\leq \Delta -1$, we get
$$ABS(G)\leq(\Delta-1)\sqrt{\frac{2\Delta-3}{2\Delta-1}}+\frac{\Delta n-2\Delta+1}{2}\sqrt{\frac{\Delta-1}{\Delta}}.$$
Moreover, the equality holds if and only if $d_v=\Delta-1$, $f(d_u,d_v)=f(\Delta-1,\Delta)$ for all $u\in N(v)$, and $f(u,w)=f(\Delta,\Delta)$ for all $uw\in E(G)\setminus M$.  Thus, the equality holds if and only if  $G$ is nearly $\Delta-$regular.
\end{proof}

A graph is called $r-$partite if its set of vertices can be partitioned into $r$ subsets, called partite sets, so that all vertices in the same partite set are pairwise non-adjacent. An $r-$partite graph is complete if every pair of vertices that belong to different partite sets are adjacent. We denote, by $T_{n,r}$, the complete $r$-partite graph of order $n$ such that $|k_i - k_j| \le 1$, where $k_i$, with $i = 1, 2, \cdots , r,$ is the number of vertices in the $i$-th partite set of $T_{n,r}$. Clearly $T_{n,r}\in \Pi_{n,r}$. The next theorem gives a sharp upper bound on the $ABS$ value of all graphs in the class $\Pi_{n,\chi}$.

\begin{theorem}\label{t1}
Let $n\geq5$ and $\chi\geq3$ and let $q= \lfloor n/\chi \rfloor$ and $r=n-q\chi$. If $G\in \Pi_{n,\chi}$ then
\begin{align}\label{eq1} ABS(G)\leq& \frac{r(r-1)q^2}{2}\sqrt{\frac{n-q-1}{n-q}}+r(\chi-r)q(q+1)\sqrt{\frac{2n-2q-3}{2n-2q-1}}\nonumber\\[2mm]
&+\frac{(\chi-r)(\chi-r-1)(q+1)^2}{2}\sqrt{\frac{n-q-2}{n-q}},\end{align}
with equality if and only if $G\cong T_{n,\chi}$.

\end{theorem}
\begin{proof}
Let $G^{\ast}\in \Pi_{n,\chi}$ be having the maximum $ABS$. The vertex set $V(G)$ of $G$ can be partitioned into $\chi$ subsets, say $Y_1,Y_2,\cdots,Y_\chi$ such that $|Y_i|=k_i$ for $i=1,2,\cdots,\chi$, provided that $k_1\leq k_2 \leq \cdots \leq k_{\chi}$. Consequently, $G^{\ast}$ is isomorphic to a $\chi$-partite graph.  Thus, by Lemma \ref{lem-AZI+e}, it must be isomorphic to the complete $\chi$-partite graph $K_{k_1,k_2,\cdots,k_{\chi}}$. It is remaining to show that $k_{\chi}-k_1\leq 1$. Seeking a contradiction, assume that $k_{\chi}-k_1\geq 2$. Let $H\cong K_{d_1,d_2,\cdots,d_{\chi}}$, where $d_1=k_1+1$, $d_{\chi}=k_{\chi}-1$, and $d_i=k_i$ for every $i\in \{2,\cdots,\chi-1\}$
\begin{align}\label{eq-01A}
ABS(H)-ABS(G^{\ast})&= (k_1+1)(k_{\chi}-1)f(n-k_1-1,n-k_{\chi}+1) - k_1k_{\chi}f(n-k_1,n-k_{\chi}) \nonumber\\
&\ \ \ \ +\sum_{i=2}^{\chi-1}\left[k_i(k_1+1)f(n-k_1-1,n-k_i) - k_1k_if(n-k_1,n-k_i)\right]\nonumber\\
&\ \ \ \ +\sum_{i=2}^{\chi-1}\left[k_i(k_{\chi}-1)f(n-k_{\chi}+1,n-k_i) - k_{\chi}k_if(n-k_{\chi},n-k_i)\right]\nonumber\\
&=(k_{\chi}-k_1-1)f(n-k_1,n-k_{\chi})\nonumber\\
&\ \ \ \ +\sum_{i=2}^{\chi-1}k_i\left[f(n-k_1-1,n-k_i) - f(n-k_{\chi}+1,n-k_i)\right]\nonumber\\
&\ \ \ \ +\sum_{i=2}^{\chi-1}k_i\left[k_{\chi}g_1(n-k_{\chi},n-k_i)-k_1g_1(n-k_1-1,n-k_i)\right]\nonumber
\end{align}
Since $n-k_1-1\geq n-k_{\chi}+1$, from Lemma \ref{l1} we get that for each $i=2,\cdots,\chi-1$,
$f(n-k_1-1,n-k_i) - f(n-k_{\chi}+1,n-k_i)\geq0$ and
$k_{\chi}g_1(n-k_{\chi},n-k_i)-k_1g_1(n-k_1-1,n-k_i)>k_1\left[g_1(n-k_{\chi},n-k_i)-g_1(n-k_1-1,n-k_i)\right]\geq0.$
So $ABS(H)-ABS(G^{\ast})>0$, a contradiction. Thus $k_{\chi}-k_1\leq 1$.
\end{proof}

The next theorem gives a sharp upper bound on the $ABS$ value of all graphs in the class $\Sigma_{n,\alpha}$.
\begin{theorem}\label{t3}
Let $n$ and $\alpha$ be positive integers. If $G\in \Sigma_{n,\alpha}$ then
$$ABS(G)\leq \alpha(n-\alpha)\sqrt{\frac{2n-\alpha-3}{2n-\alpha-1}}+\frac{1}{2}(n-\alpha)(n-\alpha-1)\sqrt{\frac{n-2}{n-1}},$$ with equality if and only if  $G\cong N_{\alpha}+K_{n-\alpha}$.
\end{theorem}
\begin{proof}
Let $G^{\ast}$ be having the maximum $ABS$ value in $\Sigma_{n,\alpha}$. Let $W$ be an independent set in $G^{\ast}$ with $|W|=\alpha$. Assume that there are two non-adjacent vertices $u$ and $v$ such that $u\in W$ and $v\in V(G^{\ast})-W$. Then $G^{\ast}+uv\in \Sigma_{n,\alpha}$ and $ABS(G^{\ast}+uv)>ABS(G^{\ast})$, a contradiction. So, each vertex in $W$ is adjacent to every vertex in $V(G^{\ast})-W$. Furthermore, every pair of vertices in $V(G^{\ast})-W$ are adjacent, yielding $G[V(G^{\ast})-W]\cong K_{n-\alpha}$. Thus  $G^{\ast}\cong N_{\alpha}+K_{n-\alpha}$.
Therefore, $$ABS(G^{\ast})= \alpha(n-\alpha)\sqrt{\frac{2n-\alpha-3}{2n-\alpha-1}}+\frac{1}{2}(n-\alpha)(n-\alpha-1)\sqrt{\frac{n-2}{n-1}}.$$
\end{proof}

 The next theorem gives a sharp upper bound on the $ABS$ value of all graphs in the class $\Gamma_{n,p}$.  We denote by $S_{n-1}$ the star of order $n$ and by $S_{m,n-m}$ the double star of order $n$, where the internal vertices have degrees $m$ and $n-m$. We also denote by $K_{m}^p$ the graph of order $m+p$ and $p$ pendent vertices such that the induced subgraph on the internal vertices is complete graph and all pendent vertices are adjacent to the same internal vertex.

\begin{theorem}\label{t5}
Let $G\in \Gamma_{n,p}$.
\begin{enumerate}
\item\label{t5p1} If $p=n-1$ then $G\cong ABS(S_{n-1})$, and thus $ABS(G)=\frac{(n-1)\sqrt{n-2}}{n}$.
\item\label{t5p2} If $p=n-2$ then $ABS(G)\leq \frac{1}{\sqrt{3}}+\frac{\sqrt{n-2}}{n}+\frac{(n-3)\sqrt{n-3}}{n-1},$  with equality if and only if $G\cong S_{2,n-2}$.
\item\label{t5p3} If $p \leq n-3$ then $$ABS(G)\leq p\sqrt{\frac{n-2}{n}}+(n-p-1)\sqrt{\frac{2n-2p-3}{2n-2p-1}}+\frac{\sqrt{n-p-1}(n-p-2)^{\frac{3}{2}}}{2},$$ with equality if and only if $G\cong K_{n-p}^p$.
\end{enumerate}
\end{theorem}

\begin{proof}
(\ref{t5p1}) Straightforward.\\
(\ref{t5p2}) Let $u$ and $v$ be the internal vertices of $G$. We may assume that there are $t$ pendent vertices adjacent to $u$ and $p-t$ pendent vertices adjacent to $v$. Thus
\begin{align}
ABS(G)&= tf(1,d_u)+(p-t)f(1,d_v)+f(d_u,d_v)\nonumber\\
              &=tf(1,t+1)+(p-t)f(1,p-t+1).\nonumber
\end{align}

Consider the function $h(t)=tf(1,t+1)+(p-t)f(1,p-t+1)$.
$$h'(t)=\frac{M-N}{(t+2)(p-t+2)\sqrt{(t+2)(p-t+2)}},$$ where $M=((p-1)t-t^2+3p+6)\sqrt{pt-t^2+2t}$ and $N=(p-t+3)(t+2)\sqrt{(p-t)(t+2)}$. Clearly, both $M>0$ and $N>0$ when $1\leq t\leq p-1$. Thus the sign of $h'(t)$ is determined by the sign of $(M-N)(M+N)=M^2-N^2$. Now
$$M^2-N^2=(2t-p)(3tp(p-t)+10t(p-t)+8p^2+48p+72).$$ Hence $h'(t)<0$ when $1\leq t<p/2$ and $h'(t)>0$ when $p/2 < t \leq p-1$. So $h(t)$ has maximum value at $t=1$ or $t=p-1$. Thus, $$ABS(G)\leq h(1)=h(p-1)=\frac{1}{\sqrt{3}}+\frac{\sqrt{p}}{p+2}+\frac{(p-1)\sqrt{p-1}}{p+1},$$   with equality if and only if $G\cong S_{2,p}=S_{2,n-2}$.

(\ref{t5p3}) Let $P$ be the set of pendent vertices in $G$.  If there are two non adjacent vertices $u,v\in V(G)\setminus P$ then $G+\{uv\}\in \Gamma_{n,p}$ and  by Lemma \ref{lem-AZI+e}, $ABS(G+\{uv\})>ABS(G)$, a contradiction. Thus the induced subgraph $G[V(G)\setminus P]\cong K_{n-p}$.
Label the vertices of  $G[V(G)\setminus P]$ by  $u_1,...,u_{n-p}$ and  for each $j=1,...,n-p$, let $d_j=|N(u_j)\cap P|$ so that $d_1\geq d_2\geq ... \geq d_{n-p}$. To obtain the desired result we want need to show that $d_1=p$ and $d_2=...=d_{n-p}=0$. So seeking a contradiction assume that $d_j\geq 1$ for some $j\geq 2$. Then $d_1\geq d_2 \geq 1$.  Let $x\in P\cap N(u_2)$ and take $G'=G-\{xu_2\}+\{xu_1\}$. Note that for each $j=1,...,n-p$, $deg_G(u_i)=d_i+n-p-1$. Then
\begin{align}
ABS(G')-ABS(G)&=f(1,d_1+n-p)-f(1,d_2+n-p-1)\nonumber\\
              &+d_1(f(1,d_1+n-p)-f(1,d_1+n-p-1))\nonumber\\
							&-(d_2-1)(f(1,d_2+n-p-1)-f(1,d_2+n-p-2))\nonumber\\
              &+\sum_{j=3}^{n-p}(f(d_j+n-p-1,d_1+n-p)-f(d_j+n-p-1,d_1+n-p-1))\nonumber\\
							&-\sum_{j=3}^{n-p}(f(d_j+n-p-1,d_2+n-p-1)-f(d_j+n-p-1,d_2+n-p-2))\nonumber\\
							&=f(1,d_1+n-p)-f(1,d_2+n-p-1)\nonumber\\
              &+d_1g_1(1,d_1+n-p-1)-(d_2-1)g_1(1,d_2+n-p-2)\nonumber\\
              &+\sum_{j=3}^{n-p}(g_1(d_1+n-p-1,d_j+n-p-1)-g_1(d_2+n-p-2,d_j+n-p-1)).\nonumber
							\end{align}
Since $d_j\geq0$ for $j=3,...,n-p$, by Lemma \ref{l2}, we have
							
\begin{align}
g_1(d_1+n-p-1,d_j+n-p-1)-&g_1(d_2+n-p-2,d_j+n-p-1)\geq \nonumber\\
&g_1(d_1+n-p-1,n-p-1)-g_1(d_2+n-p-2,n-p-1).\nonumber
\end{align}
Thus,
\begin{align}
ABS(G')-ABS(G)\geq &f(1,d_1+n-p)-f(1,d_2+n-p-1)\nonumber\\
&+d_1g_1(d_1+n-p-1,1)-(d_2-1)g_1(d_2+n-p-2,1)\nonumber\\
&+(n-p-2)(g_1(d_1+n-p-1,n-p-1)-g_1(d_2+n-p-2,n-p-1))\nonumber\\
&=w(d_1)-w(d_2-1),\nonumber
\end{align}
where $$w(t)=f(t+n-p,1)+tg_1(t+n-p-1,1)+(n-p-2)g_1(t+n-p-1,n-p-1).$$
Our next aim is to show that $w(t)$ is increasing in $t$. Note that
\begin{align}
w'(t)=&\frac{\partial f}{\partial t}(t+n-p,1)+g_1(t+n-p-1,1)+t\frac{\partial g_1}{\partial t}(t+n-p-1,1)\nonumber\\
&+(n-p-2)\frac{\partial g_1}{\partial t}(t+n-p-1,n-p-1).\nonumber
\end{align}

Since $\frac{\partial g_1}{\partial x}(x,y)$ is increasing in $y$ when $y\geq1$, we get  $$\frac{\partial g_1}{\partial t}(t+n-p-1,n-p-1)\geq \frac{\partial g_1}{\partial t}(t+n-p-1,1).$$ So
\begin{align}
w'(t)\geq& \frac{\partial f}{\partial t}(t+n-p,1)+g_1(t+n-p-1,1)+(t+n-p-1)\frac{\partial g_1}{\partial t}(t+n-p-1,1)\nonumber\\
&=L-K,\nonumber
\end{align}
where $$L=\frac{(t+n-p)^2+(t+n-p)-1}{(t+n-p-1)^{\frac{1}{2}}(t+n-p+1)^{\frac{3}{2}}}\text{ and }K=\frac{(t+n-p)^2-(t+n-p)-1}{(t+n-p-2)^{\frac{1}{2}}(t+n-p)^{\frac{3}{2}}}.$$

Since $$L^2-K^2=\frac{2(t+n-p)^5-3(t+n-p)^4-8(t+n-p)^3+3(t+n-p)^2+4(t+n-p)+1}{(t+n-p-1)(t+n-p+1)^3(t+n-p-2)(t+n-p)^3}>0,$$ we get $w'(t)\geq L-K>0$, and thus $w(t)$ is increasing in $t$ as desired. This implies that $ABS(G')-ABS(G)>w(d_1)-w(d_2-1)>0$, a contradiction. So $d_1=p$ and $d_j=0$ for all $j=2,...,n-p-1$, and hence $G\cong K_{n-p}^p$.
\end{proof}

\section*{Acknowledgement}
This work is supported by the Scientific Research Deanship, University of Ha\!'il, Saudi Arabia, through project number RG-23\,013.

\end{document}